\documentclass[a4paper,12pt]{amsart}
\usepackage{amssymb}
\usepackage{xypic}
\usepackage[all,cmtip]{xy}
\usepackage{amsmath}
\usepackage{chemarrow}
\usepackage[colorlinks, 
pdfborder=001, 
linkcolor=cyan,
anchorcolor=magenta
citecolor=green
]{hyperref}
\usepackage{mathrsfs}
\usepackage{wasysym}
\usepackage{titletoc}
\usepackage{bm}
\usepackage{geometry}
\usepackage{colordvi}
\usepackage{color}

\allowdisplaybreaks

\DeclareMathOperator{\A}{\mathcal{A}}

\DeclareMathOperator{\Aut}{Aut}

\DeclareMathOperator{\D}{\mathcal{D}}

\DeclareMathOperator{\End}{End}
\DeclareMathOperator{\Ext}{Ext}

\DeclareMathOperator{\GKdim}{GKdim}

\DeclareMathOperator{\hdet}{hdet}

\DeclareMathOperator{\Hom}{Hom}

\DeclareMathOperator{\id}{id}

\DeclareMathOperator{\kk}{\Bbbk}

\DeclareMathOperator{\Mod}{Mod}
\DeclareMathOperator{\N}{\mathbb{N}}
\DeclareMathOperator{\nr}{\mathrm{nr}}

\DeclareMathOperator{\rank}{rank}

\DeclareMathOperator{\RHom}{\textbf{R}Hom}
\DeclareMathOperator{\rhu}{\rightharpoonup}

\DeclareMathOperator{\tr}{tr}

\DeclareMathOperator{\Z}{\mathbb{Z}}

\newcommand{\To}{\longrightarrow}

\numberwithin{equation}{section}

\theoremstyle{definition}

\newtheorem{thm}{Theorem}[section]
\newtheorem{prop}[thm]{Proposition}
\newtheorem{lem}[thm]{Lemma}

\newtheorem{cor}[thm]{Corollary}
\newtheorem{defn}[thm]{Definition}
\newtheorem{rk}[thm]{Remark}

\newtheorem{ex}[thm]{Example}

\begin{document}

\title{A note on the discriminant of reflection Hopf algebras}


\author{Ruipeng Zhu}
\address{Department of Mathematics, Southern University of Science and Technology, Shenzhen, Guangdong 518055, China}
\email{zhurp@sustech.edu.cn}

\begin{abstract}
%
	We provide a formula for commputing the discriminant of skew Calabi-Yau algebra over a central Calabi-Yau algebra.
	This method is applied to study the Jacobian and discriminant for reflection Hopf algebras.
\end{abstract}


\thanks{}

\maketitle



\section*{Introduction}



Recently, the discriminant has been used to determine the automorphism group of some PI algebras \cite{CPWZ1, CPWZ3}, and solve isomorphism problem \cite{CPWZ2} and Zariski cancellation proplem \cite{BZ, LeWZ, LWZ}.
Despite the usefulness of the discriminant, the computation of the discriminant turns out to be a rather diffculty problem, see \cite{CYZ} for example.
Nguyen, Trampel and Yakimov presented a general method for computing the discriminants of algebras in \cite{NTY}, by associating discriminants of noncommutative algebras with Poisson geometry.
In \cite{LY}, Levitt and Yakimov maked use of techniques from quantum cluster algebras to derive explicit formulas for the discriminants of quantized Weyl algebras at roots of unity.

Almost all these algebras are skew Calabi-Yau algebras which are finitely generated modules over central subalgebras. It worth pointing out that these algebras are Frobenius algebras over commutative Calabi-Yau algebras.

\begin{lem}[Corollary \ref{sCY-over-com-CY-is-Frob-alg}] \label{key-lem}
	Let $R$ be an affine Calabi-Yau algebra of dimension $d$ and $A$ be a module-finite $R$-algbera. If $A$ is a skew Calabi-Yau algebra of dimension $d$, then $A$ is a Frobenius $R$-algebra.
\end{lem}

The discriminant of a Frobenius algebra $A$ is the norm of the different of $A$, see Lemma \ref{det(l-multi)}.

Throughout the rest of this paper, $\kk$ is a base field with char $\kk = 0$, and all vector spaces, algebras and Hopf algebras are over $\kk$. Unadorned $\otimes$ means $\otimes_{\kk}$ and $\Hom$ means $\Hom_{\kk}$. 
$H$ will stand for a Hopf algebra $(H, \Delta, \varepsilon)$ with bijective antipode $S$. We use the Sweedler notation $\Delta(h) = \sum_{(h)} h_1 \otimes h_2$ for all $h \in H$.
We recommend \cite{Mon} as a basic reference for the theory of Hopf algebras.

We are interested in the following algebra which comes from the noncommutative invariant theory. Let $H$ be a finite-dimensional semisimple Hopf algebra, and $A$ be a Noetherian connected graded Artin-Schelter regular domain \cite{AS87}, which are skew Calabi-Yau algebras by \cite[Lemma 1.2]{RRZ}. Suppose that $H$ acts homogeneously and inner faithfully on $A$ such that $A$ is a left $H$-module algebra.
Then $H$ is called a {\it reflection Hopf algebra} or {\it reflection quantum group} \cite[Definition 3.2]{KKZ2} if the fixed subring $A^H$ is again Artin-Schelter regular.
Kirkman and Zhang \cite{KZ}, introduced some noncommutative version of Jacobian $\mathfrak{j}_{A, H}$, reflection arrangement $\mathfrak{a}_{A, H}$ and discriminant $\delta_{A, H}$ of the $H$-action on $A$, see Definition \ref{Jacb-ref-arrang-disc}.

When $A^H$ is central in $A$ and $H$ is a dual reflection group, then $\delta_{A, H}$ and the noncommutative discriminant have the same prime radical \cite[Theorem 0.7]{KZ}.
Our main result is to relate the noncommutative discriminant to the Jacobian of a general Hopf action.

\begin{thm}[Theorem \ref{main-thm}]\label{intro-main-thm 1}
	Suppose that $H$ acts on $A$ as a reflection Hopf algebra, and that $R:=A^H$ is central in $A$. Then the discriminant $d(A, R; \tr) =_{\kk^{\times}} \mathfrak{j}_{A, H}^n$ and $\sqrt{(\delta_{A, H})} = \sqrt{(d(A, R; \tr))}$,
	where $n$ is the rank of $R$-module $A$.
\end{thm}

We also study the discriminant of the smash product.

\begin{thm}[Theorem \ref{disc-smash-prod}]\label{intro-main-thm 2}
	Suppose that $H$ acts on $A$ as a reflection Hopf algebra, and that $R:=A^H$ is central in $A$. Then the discriminant $d(A\#H, R; \tr) = \mathfrak{j}_{A, H}^{nm}$, where $n$ is the rank of $R$-module $A$ and $m = \dim_{\kk}(H)$.
\end{thm}

The organization of the paper is as follows.
In section \ref{section 1}, we go over some notation and recall some preliminary results. We also describes the discriminant of Frobenius algebra in this section.
In section \ref{section 2}, the definition and basic results about skew Calabi-Yau algebras are recalled first, and Lemma \ref{key-lem} is also proved.
The necessary definition and background for reflection Hopf algebra are in section \ref{section 3}.
In section \ref{section 4}, we prove the Theorem \ref{intro-main-thm 1} which connects the discriminant with the invariant of the Hopf action.
We prove a more general formula regarding discriminants of Hopf Galois extensions from which the Theorem \ref{intro-main-thm 2} follows in section \ref{section 5}.

\section{Preliminaries}\label{section 1}

We will employ the following notation. Let $R$ be a ring, $M$ be an $R$-$R$-bimodule, and $\mu, \nu$ be automorphisms of $R$. Let ${^{\mu}M^{\nu}}$ denote the induced $R$-$R$-bimodule such that ${^{\mu}M^{\nu}} = M$ as an Abelian group, and the module structure given by
$$a \cdot m \cdot b = \mu(a)m\nu(b), \qquad \text{ for all } a, b \in R, \, m \in {^{\mu}M^{\nu}}(=M).$$
If $\mu$ (resp. $\nu$) is the identity map of $R$, then we use $M^{\nu}$ (resp. $^{\mu}M$) instead of ${^{\mu}M^{\nu}}$.

\subsection{Hattori-Stallings trace map and discriminant}

Let $R$ be a ring, and $P$ be a finitely generated projective right $R$-module. Then there is an isomorphism
$$\varphi_P: P \otimes_R \Hom_{R}(P, R) \To \End_R(P)(:= \Hom_R(P, P)), \qquad x \otimes f \mapsto \big(y \mapsto xf(y) \big),$$
and a morphism
$$\psi_P: P \otimes_R \Hom_{R}(P, R) \To R/[R,R], \qquad x \otimes f \mapsto f(x)+[R,R],$$
where $[R, R]$ is the additive subgroup of $R$ generated by the elements $ab-ba$ with $a, b \in R$. Then the {\it Hattori-Stallings map} of $P$ is defined by
$$\tr_{P_R}: = \psi_P \circ \varphi_P^{-1} : \End_R(P) \To R/[R, R].$$
We will write it simply $\tr_P$ or $\tr$ when no confusion can arise. And Hattori-Stallings map has the following properties. 

\begin{lem}\cite{Hat, Sta}\label{HS-trace}
	Let $P$ be a finitely generated projective right $R$-module.
	\begin{enumerate}
		\item Let $x_1, \dots, x_n \in P$ and $x_1^*, \dots, x_n^* \in P^*:= \Hom_R(P, R)$ be a dual basis for $P_R$. Then
		$$\tr_P(f) = \sum\limits_{i=1}^{n} x_i^*(f(x_i)) + [R, R], \text{ for any }f \in \End_R(P).$$
		\item If $f, g \in \End_R(P)$, then $\tr_P(f+g) = \tr_P(f) + \tr_P(g)$.
		\item For any $f,g \in \End_R(P)$, then $\tr_P(gf) = \tr_P(fg)$.
        \item If $P$ is a free $R$-module of rank $n$, then $\tr_P(\id_P) = n \cdot 1_R + [R, R]$, where $\id_P$ is the identity map of $P$.
	\end{enumerate}
\end{lem}

Let $R \to A$ be a ring homomorphism such that $A$ is a finitely generated projective $R$-module.
Left multiplication defines a natural embedding $\iota: A \to \End_R(A), \, a \mapsto a_L$. The composition $\tr_A \iota$ is also called {\it Hattori-Stallings map}, and denoted by $\tr_A$.

\begin{lem}\label{composition-trace}
	Retain the notation as above. Let $P$ be a finitely generated projective right $A$-module. Then for any $f \in \End_A(P)$,
	$$\tr_{P_R}(f) = \tr_{A}(\tr_{P_A}(f)).$$
\end{lem}
\begin{proof}
	Since $\tr_{A}([A, A]) = 0$ by Lemma \ref{HS-trace} (3), then the composition map $\tr_A \circ \tr_{P_A}$ is well defined.
	Let $x_1, \dots, x_n \in P$ and $x_1^*, \dots, x_n^* \in \Hom_A(P, A)$ be a dual basis of projective module $P_A$.
	Let $y_1, \cdots, y_m \in A$ and $y_1^*, \cdots, y_m^*\in \Hom_R(A, R)$ be a dual basis of projective module $A_R$. For any $x \in P$,
	$$\sum_{i,j} x_jy_i y^*_i(x^*_j(x)) = \sum_{j=1}^m x_jx^*_j(x) = x.$$
	Hence, $\{x_jy_i\}$ and $\{y^*_ix^*_j\}$ is a dual basis of $P_R$.
	Then, for any $f \in \End_A(P)$,
	\begin{align*}
		\tr_{P_A}(f) & = \sum_{i,j} y^*_i(x^*_j(f(x_jy_i))) + [R, R] & \\
		& = \sum_{i,j} y^*_i(x^*_j(f(x_j))y_i) + [R, R] & \\
		& = \tr_A(\sum_{j=1}^m x^*_j(f(x_j))) + [R, R] & \text{by Lemma \ref{HS-trace} (1)} \\
		& = \tr_A(\tr_{P_A}(f)) + [R, R]. &
	\end{align*}
\end{proof}

In this paper, we always consider the following trace map over a commutative ring. Let $R$ be a commutative ring, and $A$ be an $R$-algebra. An $R$-linear map $\tr: A \to R$ is called a {\it trace map} from $A$ to $R$, if $\tr(xy) = \tr(yx)$ for any $x, y \in A$.
Obviously, if $A$ is a finitely generated projective $R$-module, then the Hattori-Stallings map $\tr_{A}: A \to R$ is a trace map from $A$ to $R$. The trace map $\tr_A$ is called {\it Hattori-Stallings trace map}.

Let $R$ be a commutative ring, and $R^{\times}$ be the set of invertible elements in $R$. For any $x, y \in R$, we use the notation $x =_{R^{\times}} y$ to indicate that $x = uy$ for some $u \in R^{\times}$.

\begin{defn}\label{disc-defn}
	Let $R$ be a commutative ring, $A$ be an $R$-algebra, and $\tr : A \to R$ be a trace map from $A$ to $R$. 
	Suppose that $A$ is a free $R$-module with an $R$-basis $\{ x_1, \cdots, x_m \}$. Then the {\it discriminant} of $A$ over $R$ with respect to $\tr$ is defined to be the determinant
	$$d(A, R; \tr) =_{R^{\times}} \det([\tr(x_ix_j)]^m_{i, j =1}).$$
\end{defn}

Note that the discriminant $d(A, R; \tr)$ is independent of bases of $A$ up to multiplication by an invertible element in $R$.



\subsection{Frobenius extensions}

Let's recall some definitions about Frobenius extensions.

\begin{defn}
	Let $R$ be a (graded) subring of a (graded) ring $A$. Then $R \subseteq A$ is called a {\it (graded) Frobenius extension} if
	\begin{enumerate}
		\item $A$ is a finitely generated (graded) projective right $R$-module,
		\item and there is an (graded) $R$-$A$-bimodule isomorphism $\Theta: A \stackrel{\cong}{\To} \Hom_R(A, R)$.
	\end{enumerate}
\end{defn}

Many examples and basic properties of Frobenius extensions can be found in \cite{Kad}.

Let $R \subseteq A$ be a Frobenius extension, and $x_1, \dots, x_n \in A$ and $f_1, \dots, f_n \in \Hom_R(A, R)$ be a dual basis of $A_R$.
Write $\theta = \Theta(1)$.
By definition, $\theta$ is an $R$-$R$-bimodule morphism, and there exists $y_1, \dots, y_n \in A$ such that $f_i = \Theta(y_i) = \theta \cdot y_i$. Then $\sum_{i=1}^n x_i \theta(y_ia) = a$ for any $a \in A$. It is easy to see that $\Theta(\sum_{i=1}^n \theta(ax_i)y_i) = \theta \cdot a$. So $\sum_{i=1}^n \theta(ax_i)y_i = a$ for any $a \in A$.

On the other hand, for any ring extension $R \subseteq A$, if there exists $x_1, \dots, x_n, y_1, \dots, y_n \in A$ and an $R$-$R$-bimodule morphism $\theta: A \to R$ such that for any $a \in A$,
$$\sum_{i=1}^n x_i \theta(y_ia) = a = \sum_{i=1}^n \theta(ax_i)y_i$$
then $A$ is a Frobenius extension of $R$.


Then $(\theta, x_i, y_i)$ is called a {\it Frobenius system}, $\theta$ a {\it Frobenius homomorphism}, $(x_i, y_i)$ a {\it dual basis}, and $\Theta: A \to \Hom_R(A, R), \, a \mapsto \theta \cdot a$ is a {\it right Frobenius isomorphism}. By definition,
\begin{equation}\label{Frob-trace}
	\tr: A \To R/[R, R], \qquad a \mapsto \sum_{i=1}^N \theta(y_iax_i) + [R, R]
\end{equation}
is the Hattori-Stallings trace map by Lemma \ref{HS-trace} (1).

Clearly, the definition is left-right symmetric, that is, $R \subseteq A$ is a Frobenius extension if and only if $R^{op} \subseteq A^{op}$ is a Frobenius extension, where $A^{op}$ and $R^{op}$ are the opposite rings of $A$ and $R$, respectively.

Frobenius extension is a transitive notion as follow.
\begin{lem}\label{Frob-trans-lem}
	If $R \subseteq A$ and $S \subseteq R$ are Frobenius extensions, then the composite ring extension $R \subseteq A$ is also a Frobenius extension.
\end{lem}

If $R$ is a commutative ring, then $A$ is also called a {\it Frobenius $R$-algebra}.
And there is an automorphism $\mu \in \Aut(A)$ such that the right Frobenius isomorphism
$$\Theta: {^\mu A} \To \Hom_R(A, R), \qquad a \mapsto \theta \cdot a$$
is an $A$-$A$-bimodule morphism, and $\mu$ is called the {\it Nakayama automorphism} of $A$.
The Nakayama automorphism of $A$ is not unique, but unique up to inner automorphisms of $A$.

Since the Hattori-Stallings trace map $\tr: A \to R$ is an $R$-module morphism, there exists $\omega \in A$ such that $\tr = \Theta(\omega) = \theta \cdot \omega$. 
Further, if there exists an $A$-$A$-bimodule isomorphism $\Theta: A \to \Hom_R(A, R)$, then $\tr:= \Theta(1)$ is a trace map from $A$ to $R$, and $A$ is called a {\it symmetric  Frobenius} $R$-algebra. 
It is obvious that $\End_R(P)$ is a symmetric Frobenius $R$-algebra, where $P$ is a finitely generated projective $R$-module.

Brown, Gordon and Stroppel proved that many skew Calabi-Yau algebras are Frobenius algebras over their regular central subalgebras in \cite{BGS}.

\begin{lem}\label{quotient-Frobenius}
	Let $A$ be a Frobenius $R$-algebra, and $S$ be a commutative $R$-algebra. Then $A \otimes_R S$ is a Frobenius $S$-algebra.
\end{lem}

\subsection{Different and discriminant of free Frobenius algebra}

%

%
A Frobenius $R$-algebra $A$ is called a {\it free Frobenius} $R$-algebra,  if $A$ is a free $R$-module.
\begin{defn}
	Let $A$ be a free Frobenius $R$-algebra. For any $a \in A$, we will denote by $\nr(a)$ the determinant of the left multiplication linear transformation $a_L \in \End_R(A)$. And we call $\nr(a)$ the {\it norm} of $a$.
\end{defn}

The following lemma gives a way to compute the discriminant of free Frobenius algebra.

\begin{lem}\label{det(l-multi)}
	Let $A$ be a free Frobenius $R$-algebra with a Frobenius system $(\theta, x_i, y_i)$ such that $\{x_i\}_{i=1}^n$ is a basis of free $R$-module $A$. Let $\mu$ be the Nakayama automorphism of $A$. Then
	$$\tr = \theta \cdot \omega, \quad \text{ and } \quad d(A, R; \tr) =_{R^{\times}} \nr(\omega),$$
	where $\omega = \sum \mu(x_i)y_i$ is called a {\it different} of $A$ over $R$. In particular, the different $\omega$ is a $\mu$-normal element of $A$, that is, $\mu(a) \omega = \omega a$ for any $a \in A$.
\end{lem}
\begin{proof}
	Recall that the right Frobenius isomorphism $\Theta: {^\mu A} \to \Hom_R(A, R), \, a \mapsto \theta \cdot a$ is an $A$-$A$-bimodule morphism. Hence $a \cdot \theta = \theta \cdot \mu(a)$ for all $a \in A$.
	Since $\tr(a) = \theta(\sum_{i=1}^n y_i a x_i)$ for any $a \in A$, then $\tr = \theta \cdot \omega$ where $\omega = \sum_{i=1}^n \mu(x_i)y_i$. Clearly, $\{y_i\}_{i=1}^n$ is also a basis of free $R$-module $A$. Write $\omega x_i = \sum_{j=1}^n \omega_{ij} y_j$ with $\omega_{ij} \in R$. Since $\theta(y_ix_j) = \delta_{ij}$ by definition, 
	then
	\begin{align*}
		d(A, R; \tr) & =_{R^{\times}} \det([\tr(x_ix_j)]_{i,j = 1}^n) \\ 
		& =_{R^{\times}} \det([\theta(\omega x_ix_j)]_{i,j = 1}^n) \\
		& =_{R^{\times}} \det([\sum_{k = 1}^n \omega_{ik} \theta(y_kx_j)]_{i,j = 1}^n) \\
		& =_{R^{\times}} \det([\omega_{ij}]_{i,j = 1}^n) \\
		& =_{R^{\times}} \nr(\omega).
	\end{align*}
	And, for any $a, b \in A$,
	$$(\theta \cdot (\omega a)) (b) = \tr(ab) = \tr(ba) = (a \cdot \theta \cdot \omega)(b) = (\theta \cdot (\mu(a) \omega))(b).$$
	It follows that $\Theta(\omega a) = \Theta(\mu(a) \omega)$, thus $\omega a = \mu(a) \omega$ since $\Theta$ is an isomorphism.
\end{proof}

Here is an example to show how to calculate the discriminant of free Frobenius algebras by using Lemma \ref{det(l-multi)}.

\begin{ex}
	Recall that the skew polynomial algebra $A:=\kk_{p_{ij}}[x_1, \dots, x_n]$ is generated by $x_i$ and subject to the relations $x_jx_i = p_{ij}x_ix_j$ for all $i < j$, where each $p_{ij}$ is a root of unity in $\kk$. Set $p_{ii} = 1$ for all $1\leq i \leq n$ and $p_{ji} = p_{ij}^{-1}$ for all $1 \leq i < j \leq n$. Let $l_1, \dots, l_n \in \Z_{+}$ such that $p_{ij}^{l_i} = 1$ for all $j$. Then $R:= \kk[x_1^{l_1}, \dots, x_n^{l_n}]$ is a central subalgebra of $A$ where $l_i \in \N$. It is clear that $\{ x_1^{i_1} \cdots x_n^{i_n} \}_{0 \leq i_s < l_s}$ is a basis of $R$-module $A$.
	There is an $R$-module morphism $\theta$ defined by, for any $s = 1, \dots, n$ and $0 \leq i_s < l_s$, $$\theta(x_1^{i_1} \cdots x_n^{i_n}) = \begin{cases}
		1, & i_s = l_s - 1 \text{ for all } i \\
		0, & \text{ otherelse}.
	\end{cases}$$
	Then $A$ is a free Frobenius $R$-algebra with a Frobenius system $(\theta, x_1^{i_1} \cdots x_n^{i_n}, x_1^{l_1-i_1-1} \cdots x_n^{l_n-i_n-1})$. By a direct computation, the Nakayama automorphism $\mu$ sends $x_i$ to $p_{1i} \cdots p_{i-1 \, i} p_{i+1 \, i} \cdots p_{n \, i} x_i$,
	$$\tr =_{\kk^{\times}} \theta \cdot (x_1^{l_1-1} \cdots x_n^{l_n-1}), \text{ and } d(A, R; \tr) =_{\kk^{\times}} (x_1^{l_1-1} \cdots x_n^{l_n-1})^{l_1 \cdots l_n}.$$
\end{ex}

\section{Skew Calabi-Yau algebras} \label{section 2}

First of all, we recall the definition of skew Calabi-Yau algebras.

For any algebra $A$, let 
$A^e = A\otimes A^{op}$ be its enveloping algebra. Then a right $A^e$-module $M$ is viewed the same as an $A$-$A$-bimodule by $a \cdot m \cdot b = m(b \otimes a)$.
\begin{defn} \label{defi-of-CY}
	An algebra $A$ is called {\it skew Calabi-Yau} of dimension  $d$, if 
	\begin{enumerate}
		\item $A$ is {\it homologically smooth}, that is, as an $A^e$-module, $A$ has a finitely generated projective resolution of finite length;
		\item there is some $\mu \in \Aut(A)$, such that, as an $A$-$A$-bimodule,
		\begin{align}\label{nak-aut}
			\Ext^i_{A^e}(A,A^e) \cong
			\begin{cases}
				0, & i \neq d \\
				A^{\mu}, & i = d.
			\end{cases}
		\end{align}
	\end{enumerate}
\end{defn}
This generalises Ginzburg's notion of a Calabi-Yau algebra \cite{Gin} which covers the case where $\mu = \id_A$. 
In general, $\mu$ is called a {\it Nakayama automorphism} of $A$ by Brown and Zhang in \cite{BrZ}.

As pointed out in \cite{BGS}, when both usages of the term ``Nakayama automorphism" are in play, they define the same map which is unique up to inner automorphism of $A$, see Lemma \ref{rigid-dualizing} for instance.

If $A$ is Calabi-Yau, then $\mu$ can be chosen as the identity map on $A$. 
If $A$ is commutative, then the Nakayama automorphism $\mu$ must be the identity map.
It is well known that the polynomial rings are Calabi-Yau algebras, and their skew group algebras are skew Calabi-Yau algebras, see \cite{WZ} for example.

\subsection{Rigid dualizing complexes and rigid Gorenstein algebras}

In this subsection, the definitions of rigid dualizing complexes and rigid Gorenstein algebras are recalled, and Lemma \ref{key-lem} is also proved.

Let $\D(\Mod A)$ be the derived category of right $A$-modules, and let $\D^*(\Mod A)$, for $* = b, +, -$, or blank, be the full subcategories of bounded, bounded below, bounded above, or unbounded complexes, respectively.
For any complex $X \in \D^{+}(\Mod A^e)$, there is a derived functor
$$\RHom_A(-, X) : \D(\Mod A^e) \To \D(\Mod A^e), \qquad Y \mapsto \RHom_A(Y, X).$$

The dualizing complexes over noncommutative rings were introduced by Yekutieli in \cite{Ye1}.

\begin{defn}
	Let $A$ be a (left and right) Noetherian algebra. Then an object $\Omega$ of $\D^b(\Mod A^e)$ is called a {\it dualizing complex} over $A$ if it satisfies the following conditions
	\begin{enumerate}
		\item $\Omega$ has finite injective dimension over $A$ and $A^{op}$.
		\item $\Omega$ has finitely generated cohomology modules over $A$ and $A^{op}$.
		\item The natural morphisms $\Phi: A \to \RHom_A(\Omega, \Omega)$ and $\Phi^{op}: A \to \RHom_{A^{op}}(\Omega, \Omega)$ are isomorphisms in $\D(\Mod A^e)$.
	\end{enumerate}
\end{defn}


\begin{defn}\cite[Definition 8.1]{V}
	Let $A$ be a Noetherian algebra. A dualizing complex $\Omega$ over $A$ is called {\it rigid} if 
	$$ \RHom_{A^e}(A,{}_A\Omega \otimes \Omega_A) \cong \Omega $$
	in $\D(\Mod A^e)$.
\end{defn}

A rigid dualizing complex is unique up to an isomorphism in $\D(\Mod A^e)$ if it exists.

\begin{lem}\cite[Proposition 8.2.(1)]{V}\label{rigid-dual-comp-iso}
	Let $A$ be a Noetherian algebra, and $\Omega_1, \Omega_2$ be two rigid dualizing complexes for $A$. Then $\Omega_1 \cong \Omega_2$ in $\D(\mathrm{Mod}A^e)$.
\end{lem}

\begin{defn} \cite[Definition 4.4]{BrZ}\label{defn-rigid-Gorenstein}
	A Noetherian algebra $A$ is called {\it rigid Gorenstein} of dimension $d$ if the injective dimension of $_AA$ and $A_A$ are both equal to $d$, 
	and there is an automorphism $\mu$ satisfying the condition \eqref{nak-aut}.
    The automorphism $\mu$ is also called the {\it Nakayama automorphism} of $A$.
\end{defn}

It is clear that a rigid Gorenstein algebra $A$ is skew Calabi-Yau, if $A$ is homological smooth.
By definition, the following lemma can be verified directly.

\begin{lem}\label{rigid-goren-vs-dual-comp}
	Let $A$ be a Noetherian algebra. Then $A$ is rigid Gorenstein of dimension  $d$ if and only if there is an automorphism $\mu$ of $A$ such that ${}^{\mu}A[d]$ is a rigid dualizing complex for $A$.
\end{lem}

We call an algebra $R$ {\it affine} if it is finitely generated as a $\kk$-algebra. An $R$-algebra $A$ is called a module-finite $R$-algebra if it is finitely generated as an $R$-module.

\begin{lem}\label{rigid-dual-comp-center-finite}
	Let $R$ be an affine commutative algebra, and $A$ be a module-finite $R$-algebra. If $R$ is rigid Gorenstein of dimension  $d$, then $\RHom_R(A, R[d])$ is a rigid dualizing complex for $A$.
\end{lem}
\begin{proof}
	The conclusion follows from the same proof of \cite[Proposition 5.7]{Ye2}.
\end{proof}

The following result was announced in \cite[section 2.5]{BGS}. For the convenience of reader, we present a full proof.

\begin{lem}\label{rigid-dualizing}
	Let $R$ is an affine commutative rigid Gorenstein algebra, and $A$ be a module-finite $R$-algebra. If $A$ is a Frobenius $R$-algebra, then $A$ is rigid Gorenstein such that
	$$\Hom_R(A, R) \cong {^{\mu}\!A},$$
	where $\mu$ is the Nakayama automorphism of $A$ in the Definition \ref{defn-rigid-Gorenstein}.
\end{lem}
\begin{proof}
	By Lemma \ref{rigid-dual-comp-center-finite}, $\RHom_{R}(A, R[d])$ is a rigid dualizing complex over $A$, where  $d$ is the injective dimension of $R$. By assumption, $A$ is a Frobenius $R$-algebra, that is, $A$ is a projective $R$-module and there exists an automorphism $\mu$ of $A$ such that $\Hom_R(A, R) \cong {^{\mu}A}$ as $A$-$A$-bimodules.
	Hence $\RHom_{R}(A, R[d]) \cong {^{\mu}A}[d]$ in $\D(\Mod A^e)$. Then the conclusion follows from Lemma \ref{rigid-goren-vs-dual-comp}.
\end{proof}

\begin{prop}\label{Goren-Frob}
	Let $R$ is an affine commutative Calabi-Yau algebra of dimension  $d$, and $A$ be a module-finite $R$-algebra. Then $A$ is rigid Gorenstein of dimension  $d$ if and only if $A$ is a Frobenius $R$-algebra. Further, if $A$ is Calabi-Yau, then $A$ is a symmetric Frobenius $R$-algebra.
\end{prop}
\begin{proof}
	By Lemma \ref{rigid-dualizing}, we only need to prove that $A$ is a Frobenius $R$-algebra if $A$ is rigid Gorenstein of dimension  $d$. According to Lemma \ref{rigid-goren-vs-dual-comp} and \ref{rigid-dual-comp-center-finite}, $\RHom_{R}(A, R[d]) \cong {^{\mu}A}[d]$ in $\D(\Mod A^e)$.
	Hence $\Ext^i_R(A, R) = \begin{cases}
		^{\mu}\!A, & i = 0 \\
		0, & i \neq 0.
	\end{cases}$
	Since $R$ has finite global dimension by definition, then there exists an $R$-module $N$, such that $\Ext^n_R(A, N) \neq 0$, where $n$ is the projective dimension of $A_R$. Let $F$ be a free $R$-module with a surjective $R$-module morphism $F \twoheadrightarrow N$. Then there is a short exact sequence $\Ext^n_R(A, F) \to \Ext^n_R(A, N) \to 0$. Hence $\Ext^n_R(A, F) \neq 0$. Since $R$ is a Noetherian ring and $A$ is a finitely generated $R$-module, it follows that
    $$F \otimes_R \Ext^n_R(A, R) \cong \Ext^n_R(A, F) \neq 0.$$
    So $n = 0$, that is, $A$ is a projective $R$-module.
    Thus $A$ is a Frobenius $R$-algebra.
    If $A$ is Calabi-Yau, then the Nakayama automorphism $\mu$ is inner. It follows that $A$ is a symmetric Frobenius $R$-algebra.
%
\end{proof}

\begin{cor}\label{sCY-over-com-CY-is-Frob-alg}
	Let $R$ be an affine Calabi-Yau algebra of dimension $d$ and $A$ be a module-finite $R$-algbera. If $A$ is a skew Calabi-Yau algebra of dimension $d$, then $A$ is a Frobenius $R$-algebra.
\end{cor}

As a consequence of Proposition \ref{Goren-Frob}, we infer the following corollary by using Artin-Tate lemma.

\begin{cor}\label{rigid-Goren-Frob-alg-over-poly-alg}
	Let $A$ be an affine rigid Gorenstein algebra which is module-finite over its center. Then $A$ is a Frobenius algebra over a polynomial algebra.
\end{cor}

\subsection{Artin-Schelter regular algebras}

In the following, we recall the definition of Artin-Schelter regular algebras \cite{AS87}, which are skew Calabi-Yau algebras by \cite[Lemma 1.2]{RRZ}.

An algebra $A$ is called {\it connected graded} if
$$A = A_0 \oplus A_1 \oplus A_2 \oplus \cdots$$
and $1 \in A_0 = \kk$, $A_iA_j \subseteq A_{i+j}$ for all $i, j \in \N$.
Write $A_{>0} = A_1 \oplus A_2 \oplus \cdots$ and $\kk = A/A_{>0}$.
A connected graded algebra $A$ is called {\it locally finite} if $\dim_{\kk} (A_i) < + \infty$ for all $i$.
Then the {\it Hilbert series} of $A$ is defined to be $h_A(t) = \sum_{i \in \N} \dim_{\kk} (A_i) t^i$. The Gelfand-Kirillov dimension of a connected $\N$-graded, locally finite algebra $A$ is defined to be
$$\GKdim (A) = \underset{n \to \infty}{\mathrm{lim \, sup}} \frac{\log (\sum_{i=0}^n \dim_{\kk} (A_i))}{\log (n)},$$
see \cite[Chapter 8]{MR}.

\begin{defn}
	A Noetherian connected graded algebra $A$ is called {\it Artin-Schelter Gorenstein} (or {\it AS Gorenstein}, for short) of dimension $d$, if the following conditions hold:
	\begin{enumerate}
		\item $A$ has finite injective dimension $d$ on the left and on the right,
		\item $\Ext^i_A(\kk, A) \cong \Ext^i_{A^{op}}(\kk, A) \cong \begin{cases}
			0, & i \neq d \\
			\kk(l), & i = d \\
		\end{cases}$, for some integral $l$, where $\kk:= A/A_{>0}$. 
	\end{enumerate}
	If in addition, $A$ has finite global dimension and finite Gelfand-Kirillov dimension, then $A$ is called {\it Artin-Schelter regular} (or {\it AS regular}, for short) of dimension $d$.
\end{defn}

It is well known that for a Noetherian connected graded algebra $A$, $A$ is AS Gorenstein of dimension $0$ if and only if $A$ is a Frobenius algebra.

Due to \cite[Proposition 8.4]{V}, each Noetherian connected graded AS Gorenstein algebra is rigid Gorenstein.

\begin{lem}\label{AS-gr-free-Frob}
	Let $A$ be a Noetherian connected graded AS Gorenstein algebra, $R$ be a central graded subalgebra of $A$. Suppose that $R$ is a graded polynomial algebra, and that $A$ is a finitely generated $R$-module. Then $A$ is a graded free Frobenius $R$-algebra.
\end{lem}
\begin{proof}
	By Proposition \ref{Goren-Frob}, $A$ is a Frobenius $R$-algebra. Thus $A$ is  graded projective by \cite[Corollary 2.3.2]{NVO}. Since $R$ is also a connected graded algebra, it follows that $A$ is a graded free $R$-module.
	Clearly, $\Hom_{R}(A, R)$ is a graded free $A$-module of rank one. So $A$ is a graded free Frobenius $R$-algebra.
\end{proof}

To conclude this section, we consider the following example.

\begin{ex}
	Let $A$ be the algebra $\kk \langle x, y \rangle / (y^2x-xy^2, yx^2+x^2y)$, which is a Noetherian AS regular domain of dimension 3 by \cite[(8.11)]{AS87}. It is easily seen that $A$ is finite over the central subalgebra $R$ which is generated by $x^4, y^2, z:=(xy)^2+(yx)^2$. One can check that $R$ is a polynomial algebra. Setting $\deg(x) = \deg(y) = 1$, the Hilbert series of $A$ and $R$ are
	$$h_A(t) = \frac{1}{(1-t)^2(1-t^2)}, \qquad h_R(t) = \frac{1}{(1-t^2)(1-t^4)^2}, \qquad \text{ respectively.}$$
	As an $R$-module, $A$ is free of rank 16.
	
	By Lemma \ref{AS-gr-free-Frob}, $A$ is a free Frobenius $R$-algebra. 
	Let $\Theta: {^{\mu}A} \to \Hom_R(A, R)$ be an $A$-$A$-bimodule isomorphism, and $\omega$ be the different of $A$ over $R$. Write $\theta = \Theta(1)$. Thus $\deg(\theta) = \deg(\frac{h_A(t)}{h_R(t)}) = - 6$ and $\deg(\omega) = - \deg(\theta) = 6$.
	By \cite[(E1.5.7)]{LMZ}, the Nakayama automorphism $\mu$ sends $x$ to $-x$, $y$ to $y$.
	Note that $\omega$ is a $\mu$-normal element by Lemma \ref{det(l-multi)}.
	Then, by a straightforward calculation,
	$$\omega =_{\kk^{\times}} x^2((xy)^2 - (yx)^2).$$
	
	Let $K$ be the fractional field of $R$. Then $Q:= A\otimes_R K$ is a finite-dimensional division ring over $K$ by Posner's theorem \cite[Theorem 13.6.5]{MR}.
	Let $f(X) \in K[X]$ be the characteristic polynomial of left multiplication of $\omega$ in $\End_K(Q)$.
	Clearly, $m(X) := X^2 - \omega^2 = X^2 - x^4(z^2 + 4x^4y^4)$ is the monic minimal polynomial of $\omega$ over $K$.
	Since $Q$ is a division ring, then $m(X)$ is irreducible in $K[X]$. Note that $\deg(f) = 16$. Thus $m(X)^8 = f(X)$.
	
	According to Lemma \ref{det(l-multi)}, the discriminant
	$$d(A, R; \tr) =_{\kk^{\times}} \nr(\omega) =_{\kk^{\times}} f(0) =_{\kk^{\times}} \omega^{16} = (x^4 (z^2+4x^4y^4))^8,$$
	which has already been determined in \cite[Example 5.1]{CPWZ1}.
\end{ex}

\section{Reflection Hopf algebra} \label{section 3}


In this section we will recall some concepts for Hopf algebra actions on AS Gorenstein algebras: homological determinant \cite{JZ, KKZ}, Jacobian, reflection arrangement and discriminant \cite{KZ}.

Throughout this section, let $H$ be a finite-dimensional semisimple Hopf algebra, and $A$ be a Noetherian connected graded AS Gorenstein algebra of dimension $d$. Suppose that $H$ acts homogeneously and inner faithfully on $A$ such that $A$ is a left $H$-module algebra.

\subsection{Jacobian, reflection arrangement and discriminant of the Hopf action}

Let $A\#H$ be the corresponding smash product algebra. For any two left $A\#H$-modules $M$ and $N$, we define a left $H$-action on $\Hom_{A^{op}}(M,N)$ induced by the left $H$-actions on $M$ and $N$ by
\begin{equation}\label{H-module-on-Hom}
	(h \rhu f) (m) = \sum_{(h)} h_2 \rhu f(S^{-1}(h_1) \rhu m)
\end{equation}
for all $h \in H$, $f \in \Hom_{A^{op}}(M, N)$, $m \in M$.
Then there is an induced $H$-module structure on $\Ext^d_{A^{op}}(\kk, A)$ which is one-dimension vector space.
Thus there is an unique algebra morphism $\eta: H \to \kk$ such that
\begin{equation}\label{hdet-defn}
	h \rhu e = \eta(h) e
\end{equation}
for all $h \in H$ and $e \in \Ext^d_{A^{op}}(\kk, A)$.

\begin{defn}\label{hdet-Definition}
	Let $\eta$ be defined in \eqref{hdet-defn}. The composite map $\eta \circ S: H \to \kk$ is called the {\it homological determinant} of $H$-action on $A$, and denoted by $\hdet_A$ or $\hdet$.
\end{defn}

Here our definition for homological determinant is not the original definition given in \cite{KKZ}. However, due to \cite[Lemma 5.10 (c)]{KKZ} the original definition is equivalent to Definition \ref{hdet-Definition}.

The fixed subring of the $H$-action on $A$ is defined to be
$$A^H:= \{ a \in A \mid h \rhu a = \varepsilon(h) a, \; \forall \, h \in H \}.$$
For any algebra morphism $\chi: H \to \kk$, we write
$$A^{\chi} = \{ a \in A \mid h \rhu a = \chi(h)a, \; \forall \, h \in H \}.$$
Notice that $\hdet$ can be seen as a group-like element of the dual Hopf algebra $H^*$. Then $\hdet^{-1} := \hdet \circ S$ is the inverse of $\hdet$ in $H^*$.


\begin{defn}\cite[Definition 2.1 and 3.1]{KZ}\label{Jacb-ref-arrang-disc}
	Let $A$ be AS Gorenstein, $\hdet: H \to \kk$ be the homological determinant of the $H$-action on $A$ and $R = A^H$.
	\begin{enumerate}
		\item If $A^{\hdet^{-1}}$ is free of rank one over $R$ on both sides and $A^{\hdet^{-1}} = f_{\hdet^{-1}}R = Rf_{\hdet^{-1}}$, then the {\it Jacobian} of the $H$-action on $A$ is defined to be
		$$\mathfrak{j}_{A, H} :=_{\kk^{\times}} f_{\hdet^{-1}} \in A^{\hdet^{-1}}.$$
		\item If $A^{\hdet}$ is free of rank one over $R$ on both sides and $A^{\hdet} = f_{\hdet}R = Rf_{\hdet}$, then the {\it reflection arrangement} of the $H$-action on $A$ is defined to be
		$$\mathfrak{a}_{A, H} :=_{\kk^{\times}} f_{\hdet} \in A^{\hdet}.$$
		\item Suppose that both the Jacobian $\mathfrak{j}_{A, H}$ and the reflection arrangement $\mathfrak{a}_{A, H}$ exist.
		The {\it left discriminant} and {\it right discriminant}of the $H$-action on $A$, are defined to be
		$$\delta^l_{A, H} :=_{\kk^{\times}}  \mathfrak{a}_{A, H}\mathfrak{j}_{A, H} \in R, \qquad \delta^r_{A, H} :=_{\kk^{\times}} \mathfrak{j}_{A, H}\mathfrak{a}_{A, H} \in R, \qquad \text{respectively}.$$
		If $\delta^l_{A, H} =_{\kk^{\times}} \delta^r_{A, H}$, then $\delta^l_{A, H}$ is called {\it discriminant} of the $H$-action on $A$, and denote by $\delta_{A, H}$.
	\end{enumerate}
\end{defn}

By \cite[Theorem 2.4]{KZ}, the Jacobian $\mathfrak{j}_{A, H}$ exists if and only if $A^H$ is AS Gorenstein.


\begin{defn}\cite[Definition 3.2]{KKZ2}
	Suppose that $A$ is a Noetherian AS regular domain. If the fixed subring $A^H$ is again Artin-Schelter regular, then we say that $H$ acts on $A$ as a {\it reflection Hopf algebra}.
\end{defn}

Suppose that $H$ acts on $A$ as a reflection Hopf algebra. Then $A$ is graded free over $A^H$ on both sides by \cite[Lemma 3.3]{KKZ2}. Hence there exists a integral coefficient polynomial $g(t)$ such that $h_A(t) = g(t) h_{A^H}(t)$, where $h_A(t)$ and $h_{A^H}(t)$ are the Hilbert series of $A$ and $A^H$, respectively.

We now recall some important results about Jacobian, reflection arrangement and discriminant of Hopf actions on AS regular algebras.

\begin{thm}\label{Jacob-thm}\cite[Corollary 2.5 and Theorem 3.8]{KZ}
	Suppose that $A$ is a Noetherian AS regular domain and $H$ acts on $A$ as a reflection Hopf algebra.
	\begin{enumerate}
		\item Both $\mathfrak{j}_{A,H}$ and $\mathfrak{a}_{A, H}$ exist.
		\item $\deg(\mathfrak{j}_{A,H}) = \deg (h_A(t) h_{A^H}(t)^{-1})$.
		\item The discriminant $\delta_{A, H}$ is defined, that is, $\mathfrak{a}_{A, H} \mathfrak{j}_{A, H} =_{\kk^{\times}} \mathfrak{j}_{A, H} \mathfrak{a}_{A, H}$.
		\item $\mathfrak{a}_{A, H}$ divides $\mathfrak{j}_{A, H}$, that is, there exists $a, b \in A$ such that $\mathfrak{j}_{A, H} = a \mathfrak{a}_{A, H} = \mathfrak{a}_{A, H} b$.
	\end{enumerate}
\end{thm}

It is well known that each commutative AS regular algebra is a graded polynomial algebra. Thus the following lemma is a consequence of Lemma \ref{AS-gr-free-Frob}.

\begin{lem}\label{ref-gr-free-Frob}
	Suppose that $H$ acts on $A$ as a reflection Hopf algebra. If $R:=A^H$ is a central subalgebra of $A$, then $A$ is a graded free Frobenius $R$-algebra.
\end{lem}

In \cite{KZ}, Kirkman and Zhang compared the $H$-discriminant in the noncommutative invariant theory (Definition \ref{Jacb-ref-arrang-disc}) to the noncommutative discriminant over a central subalgebra (Definition \ref{disc-defn}), see the following theorem.

\begin{thm}\cite[Theorem 3.10]{KZ}
	Suppose that $A$ is a Noetherian AS regular domain and $H$ acts on $A$ as a reflection Hopf algebra. Further assume that
	\begin{enumerate}
		\item [(a)] $H = (\kk G)^*$ where $G$ is a finite group, and
		\item [(b)] $R:= A^H$ is central in $A$.
	\end{enumerate}
    Then the ideals $(\delta_{A, H})$ and $(d(A, R; \tr))$ have the same prime radical, that is,
    $$\sqrt{(\delta_{A, H})} = \sqrt{(d(A, R; \tr))}.$$
\end{thm}
In the next section, we will prove that the above theorem is also true without assumption (a), see Theorem \ref{main-thm}.

\subsection{Reflection group}

If $H$ is a finite-dimensional cosemisimple Hopf algebra over an algebraically closed field which acts inner faithfully on a commutative domain, then $H$ must be a group algebra, see \cite[Theorem 1.3]{EW}. Hence the reflection Hopf algebras which act on commutative AS regular domain, are just the group algebras if $\kk$ is a algebraically closed field of characteristic zero.

We now consider the reflection groups which act on polynomial algebras and fix some notation that will be used in this subsection. Let $V = \kk x_1 \oplus \cdots \oplus \kk x_n$ be a finite-dimensional vector space. For any $\sigma \in \mathrm{GL}(V)$, $\sigma$ is called a {\it reflection} if $\sigma$ is of finite order and fixes a codimension one subspace of $V$. A finite subgroup $G$ of $\mathrm{GL}(V)$, is called a {\it reflection group} if it is generated by reflections.
Now we go over some notation about reflection group. By definition, each reflection in $G$ fixes some hyperplane in $V$. Let $\A(G)$ be the set of these reflection hyperplanes. For any $U \in \A(G)$, denote the stablilizer of $U$ in $G$ by $G_U = \{\sigma \in G \mid \sigma|_U = \id_U \}$. Since $\dim_{\kk} (U) = \dim_{\kk} (V) - 1$, $G_U$ is a cyclic group generated by a reflection $\sigma$. Let $e_U$ be the order of $\sigma$ and $\alpha_U \in V$ be a eigenvector of $\sigma$ with non-unit eigenvalue.


The following results about reflection groups and invariant theory are well known, for details, see \cite{K} and \cite[Chapter 6]{OT} for instance.

\begin{thm}\label{reflec-group-lem}
	Let $A$ be the symmetric algebra $S(V) \cong \kk[x_1, \cdots, x_d]$ and $G \subset \mathrm{GL}(V)$ be a reflection group that acts on $A$ as automorphsims. Then
	\begin{enumerate}
		\item The invariant ring $R:= A^G = \kk[f_1, \cdots, f_d]$ is a graded polynomial subalgebra of $A$.
		\item $A$ is a free $R$-module of rank $|G|$.
		\item The Jacobian of the $G$-action on $A$ is $\mathfrak{j}_{A,\,\kk\!G} =_{\kk^{\times}} \det\begin{pmatrix}
			\frac{\partial f_i}{\partial x_j}
		\end{pmatrix}_{i,j = 1}^d =_{\kk^{\times}} \prod \limits_{U \in \A(G)} \alpha_U^{e_U-1}$.
	    \item The reflection arrangement of the $G$-action on $A$ is $\mathfrak{a}_{A,\,\kk\!G} =_{\kk^{\times}} \prod \limits_{U \in \A(G)} \alpha_U$.
		\item The discriminant of the $G$-action on $A$ is $\delta_{A,\, \kk\!G} =_{\kk^{\times}} \prod \limits_{U \in \A(G)} \alpha_U^{e_U} \in A^G$.
	\end{enumerate}
\end{thm}
%
%
%
%
%

\section{Discriminant of reflection Hopf algebra} \label{section 4}

In this section, we compute the discriminant $d(A, A^H; \tr)$ when $H$ acts on $A$ as a reflection Hopf algebra and that $A^H$ is central in $A$. The trick of the proof is to find the different of $A$ over $A^H$ as defined in Lemma \ref{det(l-multi)}.
We will always assume that $H$ is a finite-dimensional semisimple Hopf algebra, and $A$ is a Noetherian connected graded AS regular domain. Suppose that $H$ acts homogeneously and inner faithfully on $A$ as a reflection Hopf algebra, and that $R = A^H$ is a central subalgebra of $A$. Notice that $A$ is a graded free Frobenius $R$-algebra by Lemma \ref{ref-gr-free-Frob}.

To prove the Theorem \ref{intro-main-thm 1}, we need several lemmas to find the different of $A$.

Since the ideal $R_{>0}A = AR_{>0}$ is a left $H$-submodule of $A$, then $\bar{A} := A/(R_{>0}A)$ has a natural left $H$-module algebra structure. And $\bar{A}$ is called the {\it covariant ring} of the $H$-action on $A$. 
Clearly, $\bar{A}$ is also a graded Frobenius algebra by Lemma \ref{quotient-Frobenius}. Hence $\bar{A}$ is AS Gorenstein.

\begin{lem}\label{hdet-bar-hdet}
	Let $\bar{A} := A/(R_{>0}A)$.
	Then $\hdet_A = \hdet_{\bar{A}}$.
\end{lem}
\begin{proof}
	For any left $A\#H$-module $M$, we have canonical isomorphisms induced by the adjointness
    $$\Hom_{\bar{A}^{op}}(\kk, \Hom_{R^{op}}(\kk, M)) \cong \Hom_{\bar{A}^{op}}(\kk, \Hom_{A^{op}}(\bar{A}, M)) \cong \Hom_{A^{op}}(\kk, M)$$
	where the $H$-module structure is given by \eqref{H-module-on-Hom}.
	From these isomorphisms, we obtain a convergent spectral sequence of $H$-modules
	$$E_2^{pq} = \Ext^p_{A^{op}}(\kk, \Ext^q_{R}(\kk, A)) \Longrightarrow \Ext^{p+q}_{A^{op}}(\kk, A).$$
	By hypothesis, $A$ is a free $R$-module, we have $\Ext^i_{R^{op}}(\kk, A) \cong \Ext^i_{R^{op}}(\kk, R) \otimes_R A$ for all $i \in \N$. It follows that
	$$\Hom_{\bar{A}^{op}}(\kk, \bar{A}) \cong \Ext^{d}_{A^{op}}(\kk, A)$$
	as $H$-modules. By the definition of the homological determinant, $\hdet_A = \hdet_{\bar{A}}$.
\end{proof}

Similar to \eqref{H-module-on-Hom}, $\Hom_R(A, R)$ is also an $H$-module via
$$(h \rhu f) (a) := f(Sh \rhu a),$$
for all $h \in H$, $f \in \Hom_R(A, R)$, $a \in A$.

\begin{lem}\label{lem2}
	Let $\theta$ be a homogenous generator of $\Hom_R(A, R)$ as a graded $A$-module with $l = \deg(h_{\bar{A}}(t)) = -\deg(\theta)$. Then $h \rhu \theta = \hdet_{A}(h) \theta$ for all $h \in H$.
\end{lem}
\begin{proof}
	The canonical surjective morphism $\pi: \bar{A} \to \kk = \bar{A}/\bar{A}_{>0}$ induces an injective morphism $$\pi^*: \Hom_{\bar{A}^{op}}(\kk, \bar{A}) \to \Hom_{\bar{A}^{op}}(\bar{A}, \bar{A}) = \bar{A}.$$
	Since $\bar{A}$ is an AS Gorenstein algebra of dimension 0, then $\Hom_{\bar{A}}(\kk, \bar{A}) \cong \kk(l)$. 
	Hence
	$$\dim_{\kk} (\bar{A}_l) = \dim_{\kk} (\pi^*(\Hom_{\bar{A}^{op}}(\kk, \bar{A})) ) = 1.$$
	Let $e$ be a non-zero element of $\bar{A}_l$. By definition of homological determinant, for all $h \in H$,
	\begin{equation}\label{Sh-e-hdet}
		Sh \rhu e = \hdet_{\bar{A}}(h)e.
	\end{equation}
	
	Since $h_A(t) = h_{\bar{A}}(t) h_R(t)$, then $\Hom_R(A, R) \cong A(-l)$ as graded $A$-module.
	Thus there exists an algebra homomorphism $\xi: H \to \kk$ such that $h \rhu \theta = \xi(h) \theta$ for all $h \in H$.
	Consider the canonical surjective $H$-module morphism
	$$\Upsilon: \Hom_R(A, R) \To \Hom_R(A, R) \otimes_R \kk \cong \Hom_{\kk}(\bar{A}, \kk), \qquad f \mapsto \big( \bar{a} \mapsto \overline{f(a)} \big).$$
	Write $\bar{\theta} = \Upsilon(\theta)$.
	Then
	$$\xi(h)\bar{\theta}(e) = \Upsilon(h \rhu\theta)(e) = (h \rhu \Upsilon(\theta))(e) = \bar{\theta}(Sh \rhu e) \stackrel{\eqref{Sh-e-hdet}}{=\!=} \bar{\theta}(\hdet_{\bar{A}}(h) e) = \hdet_{\bar{A}}(h) \bar{\theta}(e).$$
	Since $\bar{\theta}$ is a non-zero homogenous element of degree $-l$ in $\Hom_{\kk}(\bar{A}, \kk)$, then $ \bar{\theta}(e) \neq 0$. Hence $\xi = \hdet_{\bar{A}} = \hdet_A$ by Lemma \ref{hdet-bar-hdet}. We have thus proved the lemma.
\end{proof}


%
\begin{lem}\label{tr-H-mor}
	The Hattori-Stallings trace map $\tr: A \to R$ is a graded $H$-module morphism.
\end{lem}
\begin{proof}
	Note that $\End_R(A)$ is an $H$-module via
	$$(h \rhu F) (a) = \sum_{(h)} h_1 \rhu F(Sh_2 \rhu a),$$
	for any $h \in H$, $F \in \End_R(A)$ and $a \in A$.
	Clearly, the natural embedding $\iota: A \to \End_R(A), \, a \mapsto a_L$ is also an $H$-module morphism, and $\tr$ is a graded morphism.
	To prove that $\tr$ is a $H$-module morphism, we only need to prove that
	$$\varphi_A: A \otimes_R \Hom_{R}(A, R) \To \End_R(A), \qquad  x \otimes f \mapsto \big(y \mapsto xf(y) \big),$$
	and 
	$$\psi_A: A \otimes_R \Hom_{R}(A, R) \To R, \qquad x \otimes f \mapsto f(x),$$
	are $H$-module morphisms.
	Since for any $h \in H$, $x, y \in A$, and $f \in \Hom_R(A, R)$
	\begin{align*}
		\varphi_A(h \rhu (x \otimes f))(y)
		& = \varphi_A(\sum_{(h)} h_1 \rhu x \otimes h_2 \rhu f)(y) \\
		& = \sum_{(h)} (h_1 \rhu x) (h_2 \rhu f)(y) \\
		& = \sum_{(h)} (h_1 \rhu x) f(Sh_2 \rhu y) \\
		& = h \rhu \varphi_A(x \otimes f)(y).
	\end{align*}
	Then $\varphi_A$ is a left $H$-module isomorphism.
	Note that $S^2 = \id_H$ since $H$ is semisimple. Then
	\begin{align*}
		\psi_A(h \rhu (x \otimes f))
		& = \psi_A(\sum_{(h)} h_1 \rhu x \otimes h_2 \rhu f) & \\
		& = \sum_{(h)} (h_2 \rhu f) (h_1 \rhu x) & \\
		& = f((\sum_{(h)} S(h_2) h_1) \rhu x) & \\
		& = \varepsilon(h) f(x) & \text{ since } S^2 = \id_H \\
		& = h \rhu \psi_A(x \otimes f).
	\end{align*}
	So $\psi_A$ is a left $H$-module morphism. Therefore $\tr$ is a graded $H$-module morphism.
\end{proof}

In the following, we prove that the Jacobian $\mathfrak{j}_{A, H}$ is a different of $A$ over $R$.
\begin{lem}\label{degree-lem}
	Retain the notation in the Lemma \ref{lem2} and \ref{tr-H-mor}. Then $\tr =_{\kk^{\times}} \theta \cdot \mathfrak{j}_{A, H}$. 
\end{lem}
\begin{proof}
	Recall that $\theta$ is a homogenous generator of graded $A$-module $\Hom_R(A, R) (\cong A(-l))$. Since $\tr: A \to R$ is a graded map, then there exists a homogenous element $\omega \in A$ of degree $l$ such that $\tr = \theta \cdot \omega$. It is easy to check that $h \rhu (\theta \cdot \omega) = \sum_{(h)} (h_1 \rhu \theta) \cdot (h_2 \rhu \omega)$ for any $h \in H$. Then we have
	\begin{align*}
		\theta \cdot (h \rhu \omega) & = \sum_{(h)} \hdet^{-1}(h_1) \hdet(h_2) \theta \cdot (h_3 \rhu \omega) & \\
		& = \sum_{(h)} \hdet^{-1}(h_1) (h_2 \rhu \theta) \cdot (h_3 \rhu \omega) & \text{ by Lemma } \ref{lem2} \\
		& = \sum_{(h)} \hdet^{-1}(h_1) h_2 \rhu (\theta \cdot \omega) & \\
		& = \sum_{(h)} \hdet^{-1}(h_1) h_2 \rhu \tr & \\
		& = \hdet^{-1}(h) \tr & \text{ by Lemma } \ref{tr-H-mor} \\
		& = \theta \cdot (\hdet^{-1}(h)\omega). & 
	\end{align*}
	It follows that $\omega \in A^{\hdet^{-1}} = R \mathfrak{j}_{A, H}$.
	According to Theorem \ref{Jacob-thm} (2), $\deg(\mathfrak{j}_{A,H}) = l$. Thus $\omega =_{\kk^{\times}} \mathfrak{j}_{A, H}$, that is, $\tr =_{\kk^{\times}} \theta \cdot \mathfrak{j}_{A,H}$.
\end{proof}

Now we are ready to prove Theorem \ref{intro-main-thm 1}.

\begin{thm}\label{main-thm}
	Suppose that $H$ acts on $A$ as a reflection Hopf algebra, and that $R:=A^H$ is central in $A$. Then
	$$d(A, R; \tr) =_{\kk^{\times}} \mathfrak{j}_{A, H}^n \qquad \text{ and } \qquad \sqrt{(\delta_{A, H})} = \sqrt{(d(A, R; \tr))},$$
	where $\tr$ is the Hattori-Stallings trace map, and $n$ is the rank of $R$-module $A$.
\end{thm}
\begin{proof}
	Without loss of generality, assume that $H \neq \kk$ and $n > 1$.
	By Lemma \ref{det(l-multi)} and \ref{degree-lem},
	$$d(A, R; \tr) =_{\kk^{\times}} \nr(\mathfrak{j}_{A, H}) = \det((\mathfrak{j}_{A, H})_L).$$ 
	
	Let $K$ be the fractional field of $R$. Then $Q:= Q(A) = A\otimes_R K$ is a finite-dimensional division ring over $K$ by Posner's theorem \cite[Theorem 13.6.5]{MR}.
	Let $f(X) \in K[X]$ be the characteristic polynomial of left multiplication of ${\mathfrak{j}_{A,H}}$ in $\End_K(Q) = M_{n}(K)$, and
	$$m(X) = X^r + \lambda_{r-1}X^{r-1} + \cdots + \lambda_0 \in K[X]$$
	be the monic minimal polynomial of $\mathfrak{j}_{A,H}$ over $K$.
	Since $K[\mathfrak{j}_{A, H}]$ is a subfield of $Q$, it follows that $m(X)$ is irreducible in $K[X]$. Then there is some $k > 0$ such that $m(X)^k = f(X)$.
	Notice that $\hdet$ is a group-like element of the dual Hopf algebra $H^*$. Then there exists some $s > 0$ such that $\hdet^s = \varepsilon$ and $\mathfrak{j}_{A, H}^s \in R$, where $\varepsilon$ is the counit of $H$. 
	Hence $X^s - \mathfrak{j}_{A, H}^s = m(X)^t$ for some $t > 0$. Recall that $\mathrm{char} \kk = 0$.
	By comparing their coefficients, it follows that
	$$t = 1, \; s = r, \text{ and } \lambda_i = \begin{cases}
		0, & i = 1, \dots, r-1 \\
		-\mathfrak{j}_{A, H}^r, & i = 0.
	\end{cases}$$
	Note that $\deg(f(X)) = \rank_R(A) = n$. Hence $n = rk$, and the discriminant of $A$ is
	$$d(A, R; \tr) =_{\kk^{\times}} \det(\mathfrak{j}_{A, H}) =_{\kk^{\times}} f(0) =_{\kk^{\times}} m(0)^k =_{\kk^{\times}} (-\mathfrak{j}_{A, H}^r)^{k} =_{\kk^{\times}} \mathfrak{j}_{A, H}^n.$$
	
	By Definition \ref{Jacb-ref-arrang-disc}, $d(A, R; \tr)$ divides $\delta_{A, H}^n$ in $A$, that is, there exists $a, b \in A$ such that $d(A, R; \tr) = a \delta_{A, H}^n = \delta_{A, H}^n b$. Since $A$ is a domain and $d(A, R; \tr), \delta_{A, H}^n \in R$, then $a, b \in R$.
	By Theorem \ref{Jacob-thm} (4), $\mathfrak{a}_{A, H}$ divides $\mathfrak{j}_{A, H}$ in $A$. This implies that $\delta_{A, H}$ divides $d(A, R; \tr)$ in $A$, also in $R$.
	So the conclusion is verified.
\end{proof}

\begin{rk}
	By \cite[Proposition 1.8.(1)]{KWZ}, the rank of free $R$-module $A$ is equal to $\dim_{\kk}(H)$ if $A\#H$ is prime in Theorem \ref{main-thm}.
\end{rk}

\begin{ex}\label{disc-ex}
	Let $\kk = \mathbb{C}$. Fix an integral $n \geq 2$. Let $G = \langle x \rangle \times \langle y \rangle$ be the direct product of two cyclic groups of order $n$, and let $p = -e^{\frac{\pi \sqrt{-1}}{n}}, q = e^{\frac{2\pi \sqrt{-1}}{n}}$ in $\kk$. 
	Let $\sigma$ denote the automotphsim of $\kk G$ given by $\sigma(x^iy^j) = x^jy^i$. Define the semisimple Hopf algebra $H_{2n^2}$ (\cite[section 2]{Pan}) as the factor ring of the Ore extension of $\kk G$:
	$$H_{2n^2} = \frac{\kk G[z; \sigma]}{(z^2 - (\frac{1}{n}\sum_{i,j=0}^{n-1} q^{-ij} x^iy^j))},$$
	where the Hopf structure of $\kk G$ is extended to $H_{2n^2}$ by setting:
	$$\Delta(z) = \frac{1}{n}\sum_{s,t=0}^{n-1} q^{-st} x^sz \otimes y^tz, \qquad \varepsilon(z) = 1, \qquad \text{and} \qquad S(z) = z.$$

	For any $0 \leq i, j\leq n-1$, the skew polynomial ring $A:= {\kk \langle u, v \rangle} / {(vu - p^{i^2-j^2}uv )}$ is a left $H$-module algebra via
	$$x \rhu u = q^i u, \quad x \rhu v = q^j v, \quad y \rhu u = q^j u, \quad y \rhu v = q^i v, \quad z \rhu u = q^{ij} v , \quad z \rhu v = u.$$
	The homological determinant $\hdet: H \to \kk$ is given by
	$$\hdet(x) = q^{i+j}, \qquad \hdet(x) = q^{i+j}, \qquad \hdet(z) = - p^{(i+j)^2}.$$
	The $H_{2n^2}$-action on $A$ is inner faithfully if and only if $(i^2-j^2, n) = 1$ (\cite[Theorem 3.7]{FKMW}).
	Now assume that $(i^2-j^2, n) = 1$, $n$ is odd and $i-j$ is even. Then the fixed subring $A^H$ is the polynomial ring $\kk [u^n + v^n, u^nv^n]$ (\cite[Theorem 3.10]{FKMW}). This implies that $H_{2n^2}$ acts on $A$ as a reflection Hopf algebra. It is clear that $R:=A^H$ is a central subalgebra of $A$.
    By a straightforward calculation,
    $$\mathfrak{j}_{A, H_{2n^2}} =_{\kk^{\times}} u^{n-1}v^{n-1}(u^n-v^n), \qquad \mathfrak{a}_{A, H_{2n^2}} =_{\kk^{\times}} uv(u^n-v^n).$$
    It follows that
    $$d(A, A^H; \tr) =_{\kk^\times} (uv)^{2(n-1)n^2} (u^n-v^n)^{2n^2}, \qquad \delta_{A, H_{2n^2}} =_{\kk^{\times}} u^nv^n(u^n-v^n)^2.$$
\end{ex}

We now turn to the formula for the discriminant from classical commutative invariant theory.
Applying the Theorem \ref{main-thm}, we obtain the following corollary.

\begin{cor}\label{Di-fixed-ring}
	Let $A$ be the symmetric algebra $S(V)$, $G \subset \mathrm{GL}(V)$ be a finite reflection group and $R = A^G$. Then
	$$d(A,R;\tr) =_{\kk^{\times}} \mathfrak{j}_{A,\,\kk\!G}^{|G|}.$$
\end{cor}

Let $\mathcal{S}_n$ be the symmetric group acting on $A := \kk[x_1, \cdots, x_n]$ as permutations of $\{x_i\}_{i=1}^n$ and $A^{\mathcal{S}_n}$ be the invariant subring.
It is a well-known that $A^{\mathcal{S}_n}$ is generated by 
$$x_1 + x_2 + \cdots + x_n, \qquad x_1^2 + x_2^2 + \cdots + x_n^2, \qquad \cdots, \qquad x_1^n + x_2^n + \cdots + x_n^n$$
as an algebra.
Then by Theorem \ref{reflec-group-lem}, the Jacobian of $\mathcal{S}_n$-action on $A$ is
$$\mathfrak{j}_{A,\,\kk\!\mathcal{S}_n} = \det
\begin{pmatrix}
	1 & x_1 & x_1^2 & \cdots & x_1^{n-1} \\
	1 & x_2 & x_2^2 & \cdots & x_2^{n-1} \\
	1 & x_3 & x_3^2 & \cdots & x_3^{n-1} \\
	\vdots & \vdots & \vdots & \ddots & \vdots \\
	1 & x_{n} & x_n^{2} & \cdots & x_{n}^{n-1} \\
\end{pmatrix} = \prod_{i>j}(x_i-x_j).$$


It was asked in \cite[Question 4.7.]{GKM} that whether the discriminant of $A$ over $A^{\mathcal{S}_n}$ respect to $\tr$ is $\mathfrak{j}_{A,\,\kk\!\mathcal{S}_n}^{n!}$. This question was resolved by Liu and Du in \cite{LD}. It also can be seen as a special case of Corollary \ref{Di-fixed-ring}.

\begin{cor}\label{Sym-di-fixed-ring}\cite[Theorem 1.2]{LD}
	$d(A, A^{\mathcal{S}_n}; \tr) =_{\kk^{\times}} \prod\limits_{i>j}(x_i-x_j)^{n!}$.
\end{cor}

\section{Discriminant of Hopf-Galois extension} \label{section 5}

In this section, we study the discriminant of Hopf Galois extension, which is applied to the case of smash product.

Let's recall the definition of Hopf Galois extension.

\begin{defn}
	Let $H$ be a Hopf algebra and $B$ be a right $H$-comodule algebra. Then the ring extension $B^{\mathrm{co}H} := \{ b \in B \mid \sum_{(b)} b_0 \otimes b_1 = b \otimes 1 \} \subseteq B$ is $H$-{\it Galois} if the morphism
	$$\beta: B \otimes_{B^{\mathrm{co}H}} B \To B \otimes H, \qquad b' \otimes b \mapsto \sum_{(b)} b'b_0 \otimes b_1$$
	is bijective.
\end{defn}

In this section, let $H$ be a finite-dimensional Hopf algebra and $A:=B^{\mathrm{co}H} \subseteq B$ be an $H$-Galois extension. Let $0 \neq \alpha$ be a left integral of $H^*$, that is, $\sum_{(h)} h_1 \langle \alpha, h_2 \rangle = \langle \alpha, h \rangle $ for all $h \in H$. Recall that $H^*$ is a right $H$-module via $\leftharpoondown$, as follows: if $f \in H^*$, and $h,k \in H$, then $\langle f\!\leftharpoondown\!h, k \rangle = \langle f, kS(h) \rangle$, and there is a left $H$-module isomorphism $H \to H^*, \; h \, \mapsto \, \alpha\!\leftharpoondown\!h$.
Choose $t \in H$ which is satisfying that $\alpha\!\leftharpoondown t = \varepsilon$. Hence $t$ is a right integral of $H$, that is, $th = \varepsilon(h)t$ for all $h \in H$. Notice that for any $h \in H$,
\begin{equation}\label{integral-separable-idempotent}
	\sum_{(t)} h S(t_1) \otimes t_2 = \sum_{(t), (h)} h_3 S(t_1) \otimes t_2 S^{-1}(h_2) h_1 = \sum_{(t), (h)} S(t_1 S^{-1}(h_3)) \otimes t_2 S^{-1}(h_2) h_1 = \sum_{(t)} S(t_1) \otimes t_2 h.
\end{equation}

Let $x_i,y_i \in B$ such that
$$\beta(\sum_{i=1}^n x_i \otimes y_i) = \sum_{i=1}^n \sum_{(y_i)} x_i(y_i)_0 \otimes (y_i)_1 = 1 \otimes t,$$
and there is an $A$-$A$-bimodule morphism $$\theta: B \To A, \qquad b \, \mapsto \, \alpha \rhu b := \sum_{(b)} b_0 \langle \alpha, b_1 \rangle.$$

The $H$-Galois extension $A \subseteq B$ is a Frobenius extension, which is proved by Kreimer and Takeuchi (\cite[Theorem 1.7]{KT}).

\begin{lem} \label{H-Galois-Frob}
	Then $A \subseteq B$ is a Frobenius extension with a Frobenius system $(\theta, x_i, y_i)$.
\end{lem}
\begin{proof}
	Clearly, we only need to prove that $(\theta, x_i, y_i)$ is a Frobenius system. For all $b \in B$,
	\begin{align*}
	\sum_{i=1}^n \theta(bx_i)y_i & = \sum_{i=1}^n \sum_{(b), (x_i)} b_0 (x_i)_0 \langle \alpha, b_1(x_i)_1 \rangle y_i \\
	& = \sum_{i=1}^n \sum_{(b), (x_i), (y_i)} b_0 (x_i)_0(y_i)_0 \langle \alpha, b_1(x_i)_1 (y_i)_1 S(y_i)_2 \rangle \\
	& = \sum_{(b)} b_0 \langle \alpha, b_1S(t) \rangle = \sum_{(b)} b_0 \langle \alpha \leftharpoondown t, b_1 \rangle \\
	& = b.
	\end{align*}
    We also have that
    \begin{eqnarray}\label{H^*-integral-H}
    	\begin{split}
    		\langle \alpha, t \rangle & = \langle \alpha \leftharpoondown t, 1 \rangle \langle \alpha, t \rangle = \langle \alpha, S(t) \rangle \langle \alpha, t \rangle = \langle \alpha, S(t)\langle \alpha, t \rangle \rangle = \langle \alpha, S(t) S(\langle \alpha, t \rangle) \rangle \\
    		& = \sum_{(t)} \langle \alpha, S(t) S(t_1 \langle \alpha, t_2 \rangle) \rangle = \sum_{(t)} \langle \alpha, S(t)S(t_1) \rangle \langle \alpha, t_2 \rangle \\
    		& \stackrel{\eqref{integral-separable-idempotent}}{=\!=} \sum_{(t)} \langle \alpha, S(t_1) \rangle \langle \alpha, t_2S(t) \rangle = \sum_{(t)} \langle \alpha, S(t_1) \rangle \langle \alpha \leftharpoondown t, t_2 \rangle \\
    		& = \langle \alpha, S(t) \rangle \\
    		& = 1.
    	\end{split}
    \end{eqnarray}
    It follows that for all $b \in B$,
    $$\sum_{i=1}^n x_i \theta(y_ib) = \sum_{i=1}^n\sum_{(b), (y_i)} x_i (y_i)_0 b_0 \langle \alpha, (y_i)_1b_1 \rangle = \sum_{(b)} b_0 \langle \alpha, tb_1 \rangle = b \langle \alpha, t \rangle = b.$$
\end{proof}

The following lemma about the trace map of separable algebra is well known.

\begin{lem}\cite[Theorem 9.31]{R}\label{left-right-mul}
	Let $\Lambda$ be a separable algebra. Then $\tr_{\Lambda}(x_L) = \tr_{\Lambda}(x_R)$ for any $x \in \Lambda$, where $x_L$ and $x_R$ are the left and right multiplications by $x$, respectively.
\end{lem}

Now we have a result about the Hattori-Stallings map of Hopf Galois extensions.

\begin{lem}\label{trace-Galois-extension}
	Let $H$ be a finite-dimensional semisimple Hopf algebra, and $B$ be a right $H$-comodule algebra. Suppose that $A:= B^{\mathrm{co}H} \subseteq B$ is an $H$-Galois extension. Let $\tr_{B_A}: B \to A/[A,A]$ be the Hattori-Stallings map. Then
	$$\tr_{B_A}(b) + [B, B] = (\varepsilon(t) \alpha) \rhu b + [B, B] \qquad \text{ in } \; B/[B, B].$$
	Further, if $B$ is prime and $A$ is central in $B$, then $\tr_{B_A}(b) = (\varepsilon(t) \alpha) \rhu b$ for all $b \in B$.
\end{lem}
\begin{proof}
	Note that $h \stackrel{\eqref{H^*-integral-H}}{=\!=} hS(\langle \alpha, t \rangle) = \sum_{(t)} hS(t_1 \langle \alpha, t_2 \rangle) \stackrel{\eqref{integral-separable-idempotent}}{=\!=} \sum_{(t)} S(t_1) \langle \alpha, t_2h \rangle$.
	For any $h_i \in H$ and $f_i \in H^*$, write $F = \varphi_H(\sum_{i=1}^n h_i \otimes f_i) \in \End_{\kk}(H)$, thus
	\begin{equation}\label{trace-Hopf-alg}
		\begin{split}
			\tr_{H}(F) & = \sum_{i=1}^n \langle f_i, h_i \rangle = \sum_{i=1}^n \sum_{(t)} \langle f_i, S(t_1) \langle \alpha, t_2h_i \rangle \rangle \\
			& = \sum_{i=1}^n \sum_{(t)} \langle \alpha, t_2h_i\langle f_i, S(t_1) \rangle \rangle \\
			& =  \sum_{(t)} \langle \alpha, t_2F(St_1) \rangle.
		\end{split}
	\end{equation}
	
	Since $H$ is semisimple, then $S^2 = \id_H$ and $H$ is a separable algebra.
	Therefore
	$\tr_{H}(h_L) = \tr_{H}(h_R)$ by Lemma \ref{left-right-mul}, where $h_L, h_R \in \End_{\kk}(H)$ are the left multiplication and right multiplication by $h \in H$, respectively.

	For any $b \in B$,
	\begin{align*}
		\tr_{B_A}(b) + [B, B] & \stackrel{\eqref{Frob-trace}}{=\!=} \sum_{i=1}^N \theta(y_ibx_i) + [B, B] \\
		& = \sum_{i=1}^N \sum_{(b)} (y_i)_0 b_0 (x_i)_0 \langle \alpha, (y_i)_1 b_1 (x_i)_1 \rangle + [B, B] \\
		& = \sum_{i=1}^N \sum_{(b)} b_0 (x_i)_0 (y_i)_0 \langle \alpha, (y_i)_1 b_1 (x_i)_1 \rangle + [B, B] \\
		& = \sum_{i=1}^N \sum_{(b)} b_0 (x_i)_0 (y_i)_0 \langle \alpha, (y_i)_3 b_1 (x_i)_1 (y_i)_1S(y_i)_2 \rangle + [B, B] \\
		& = \sum_{(\alpha)} \sum_{(b)} b_0 \langle \alpha, t_2 b_1 S(t_1) \rangle + [B, B] \\
		& \stackrel{\eqref{trace-Hopf-alg}}{=\!=} \sum_{(b)} b_0 \tr_{H}((b_1)_L) + [B, B] \\
		& = \sum_{(b)} b_0 \tr_{{H}}((b_1)_R) + [B, B] \\
		& \stackrel{\eqref{trace-Hopf-alg}}{=\!=} \sum_{(b), (t)} b_0 \langle \alpha, t_2S(t_1)b_1 \rangle + [B, B] \\
		& = (\varepsilon(t)\alpha) \rhu b + [B, B].
	\end{align*}

	If $B$ is a prime ring and $A$ is central in $B$, then $B$ is a PI prime ring. This implies that the quotient ring $Q(B)$ of $B$ is a central simple algebra by Posner's theorem \cite[Theorem 13.6.5]{MR}.
	It follows that $[B, B] \cap A = 0$. So $\tr_{B_A}(b) = (\varepsilon(t) \alpha) \rhu b$ for any $b \in B$.
\end{proof}


The second conclusion of Lemma \ref{trace-Galois-extension} is false when $A$ is not a central subalgebra of $B$, see following example which was given in \cite{KKZ}.

\begin{ex}
	Consider the Example \ref{disc-ex} with $n = 2$. Clearly $t = (1+x)(1+y)(1+z)$ is an integral of $H$ with $\varepsilon(t) = 8$.
	Recall that $B = \kk\langle u, v \rangle / (vu - iuv)$ is a left $H$-module algebra via
	$$x \rhu u = - u, \quad x \rhu v = v, \quad y \rhu u = u, \quad y \rhu v = -v, \quad z \rhu u = v , \quad z \rhu v = u.$$
	By \cite[Example 7.4]{KKZ}, the fixed subring $B^H$ is the polynomial ring $\kk[u^2+v^2, u^2v^2]$. Hence $H$ acts on $B$ as a reflection Hopf algebra. But $B^H$ is not a central subalgebra of $B$.
	It is clear that $B$ is a free $B^H$-module of rank $8$, with a basis $\{ 1, u, v, u^2, uv, u^3, u^2v, u^3v\}$.
	Notice that the quotient ring $Q(B)$ is also an $H$-module algebra by \cite[Theorem 2.2]{SVO}. According to \cite[Theorem 8.3.7]{Mon}, it follows that $Q(B^H) = Q(B)^H \subseteq Q(B)$ is an $H^*$-Galois extension. By a direct computation, $\tr(u^2) \in [Q(B)^H, Q(B)^H] = 0$, but $t \rhu u^2 = 4(u^2+v^2) \neq 0$.
\end{ex}

Next we provide a formula for computing the discriminant of Hopf Galois extension over a central subalgebra.

\begin{lem}\label{Hopf-Galois-disc}
	Let $H$ be a finite-dimensional semisimple Hopf algebra, $A:= B^{\mathrm{co} H} \subseteq B$ be a $H$-Galois extension, and $R\;(\subseteq A)$ be a central subalgebra of $B$. Suppose that $A$ is a finitely generated free $R$-module, and that $B$ is a finitely generated free right $A$-module of rank $m$. Then
	$$d(B, R; \tr_{B_R}) =_{\kk^{\times}} d(A, R; \tr_{A_R})^m.$$
\end{lem}
\begin{proof}
	Let $t$ be a right integral of $H$ with $\varepsilon(t) = 1$, and $\alpha$ be a left integral of $H^*$ with $\langle \alpha, t \rangle = 1$. Let $\{ x_1, \dots, x_m \}$ be a basis of right $A$-module $B$. Since $A \subseteq B$ is an $H$-Galois extension, then there exists $y_1, \dots, y_m \in B$ such that $\beta(\sum_{i=1}^m x_i \otimes y_i) = 1 \otimes t$. Recall that the map $\theta$ is defined by $\theta(b) = \alpha \rhu b$.
	By Lemma \ref{H-Galois-Frob}, $B \subseteq A$ is a Frobenius extension with a Frobenius system $(\theta, x_i, y_i)$, that is,
	$$\sum_{i=1}^m x_i \theta (y_i b) = b = \sum_{i=1}^m \theta (bx_i)y_i \text{ for all }b \in B.$$
	It follows that $B$ is a free $A$-module with a basis $\{ y_i \}_{i=1}^m$, and $\theta(y_ix_j) = \delta_{ij}$ for all $1 \leq i, j \leq m$.
	Then
	$$\tr: B \To A, \qquad b \mapsto \sum_{i=1}^m \theta(y_ibx_i) + [A, A] = \sum_{i=1}^m \alpha \rhu (y_ibx_i) + [A, A],$$
	is the Hattori-Stallings map $\tr_{B_A}$.
	By Lemma \ref{trace-Galois-extension}, for all $1 \leq i, j \leq m$,
	\begin{equation}\label{eq1}
		\tr_{B_A}(y_ix_j) + [B, B] = \alpha \rhu (y_ix_j) + [B, B] = \theta(y_ix_j) + [B, B] = \delta_{ij} + [B, B], \; \text{ in } B/[B, B].
	\end{equation}

	Since $B \cong A^{\oplus m}$ as right $R$-modules, then $\tr_{A_R}(a) = \frac{1}{m} \tr_{B_R}(a) = 0$ for any $a \in A \cap [B, B]$. For any $b \in B$, since $\tr_{B_A}(b) - \theta(b) \in A \cap [B, B]$ by Lemma \ref{trace-Galois-extension}, then
	\begin{equation}\label{eq2}
		\tr_{A_R}(\tr_{B_A}(b)) = \tr_{A_R}(\theta(b)) = \tr_{A_R}(\alpha \rhu b).
	\end{equation}

	Let $z_1, \dots, z_n \in A$ be a basis of free $R$-module $A$.
	For any $1 \leq i_1, j_1 \leq m$ and $1 \leq i_2, j_2 \leq n$, 
	\begin{align*}
		\tr_{B_R}(z_{i_2}x_{i_1} y_{j_1}z_{j_2}) & = \tr_{A_R}(\tr_{B_A}(z_{i_2}x_{i_1} y_{j_1}z_{j_2})) \\
		& = \tr_{A_R}(\tr_{B_A}(x_{i_1} y_{j_1}z_{j_2}z_{i_2})) \\
		& \stackrel{\eqref{eq2}}{=\!=} \tr_{A_R}(\alpha \rhu (x_{i_1} y_{j_1}z_{j_2}z_{i_2})) \\
		& = \tr_{A_R}((\alpha \rhu (x_{i_1} y_{j_1}))z_{j_2}z_{i_2}) \\
		& \stackrel{\eqref{eq1}}{=\!=} \delta_{i_1j_1} \tr_{A_R}(z_{j_2}z_{i_2}).
	\end{align*}
	Hence $\det([\tr_{B_R}(z_{i_2}x_{i_1} y_{j_1}z_{j_2})]_{(i_1, i_2), (j_1, j_2)}) = \det([\tr_{A_R}(z_{j_2}z_{i_2})]_{i_2, j_2 = 1}^n)^m$. Since both $\{ z_{i_2}x_{i_1} \}$ and $\{ y_{j_1}z_{j_2} \}$ are basis of $R$-module $B$, then
	$$d(B, R; \tr_{B_R}) =_{\kk^{\times}} d(A, R; \tr_{A_R})^m.$$
\end{proof}

Now let's prove the Theorem \ref{intro-main-thm 2}. Retain the notation in the Theorem \ref{main-thm}. Since $A^H$ is a central subalgebra of $A$, it follows that $A^H$ is also a central subalgebra of the smash product $A \# H$.
We infer the following theorem by using Lemma \ref{Hopf-Galois-disc}.

\begin{thm}\label{disc-smash-prod}
	Suppose that $H$ acts on $A$ as a reflection Hopf algebra, and that $R:=A^H$ is central in $A$. Let $n$ be the rank of $R$-module $A$ and $m = \dim_{\kk}(H)$. Then
	$$d(A\#H, R; \tr) =_{\kk^{\times}} \mathfrak{j}_{A, H}^{nm}.$$
\end{thm}


\section*{Acknowledgments} The author is very grateful to Professor James Zhang for sharing his joint paper \cite{KWZ} with E. Kirkman and R. Won; and especially to his advisor Quanshui Wu for enlightening discussions and unceasing encouragement.


\end{document}